\documentclass[11pt]{amsart}
\usepackage{amsmath, amssymb, amsthm, mathrsfs, setspace}
\usepackage[dvips]{color} 
\usepackage[dvips]{graphicx}
\theoremstyle{plain}
\newtheorem{theorem}{Theorem}[section]
\newtheorem{proposition}[theorem]{Proposition}

\newtheorem*{acknowledge}{Aknowledgments}
\theoremstyle{definition}
\newtheorem{definition}[theorem]{Definition}

\newtheorem{remark}[theorem]{Remark}
\newcommand{\del}{\partial}

\newcommand{\Z}{\mathbb{Z}}

\begin{document}

   \title[A note on Mazur type Stein fillings of planar contact manifolds]{ A note on Mazur type Stein fillings of \\ planar contact manifolds}
   \author[Takahiro Oba]{Takahiro Oba}
   \address{Department of Mathematics, Tokyo Institute of Technology, 2-12-1 Ookayama, Meguroku, Tokyo 152-8551, Japan}
   \email{oba.t.ac@m.titech.ac.jp}
   
   
   \begin{abstract}
	We construct a family of Stein fillable contact homology $3$-spheres $\{ (M_{n}, \xi_{n}) \}_{n\geq 1}$ 
   such that the contact structure $\xi_{n}$ is supported by an open book with planar page, and 
   a Stein filling of $(M_{n},\xi_{n})$ is of Mazur type for each $n$.
   \end{abstract}
   \date{\today}

   \maketitle
  
   \section{Introduction.}
	Let $M$ be a closed, oriented, connected $3$-manifold. 
	A $2$-plane field $\xi$ on $M$ is called a {\it{contact structure}} on $M$ 
	if it is represented 
	as $\xi = {\rm{ker}}\, \alpha$ for some $1$-form $\alpha$ on $M$ satisfying $\alpha \wedge d \alpha >0$. 
	An open book of $M$ is called a {\it{supporting open book}} for $\xi = {\rm{ker}}\, \alpha$ if $d\alpha$ is an area form of each page 
	and $\alpha >0$ on its binding. Giroux \cite{G} (cf. \cite{E}) showed that there is a one-to-one correspondence between 
	contact structures on $M$ up to isotopy and open books of $M$ up to an equivalence called \textit{positive stabilization}. 
	
	Using this correspondence, Etnyre and Ozbagci \cite{EO} introduced the following invariants 
	of contact structures by their supporting open books. 
	For a contact $3$-manifold $(M, \xi)$, the {\it{support genus}} ${\rm{sg}}(\xi)$ of $\xi$ is the minimal genus of a page of 
	a supporting open book for $\xi$, and the {\it{binding number}} ${\rm{bn}}(\xi)$ of $\xi$ is the minimal number of binding components 
	of a supporting open book for $\xi$ which has a page of genus ${\rm{sg}}(\xi)$. 
	They classified the contact structures $\xi$ on $M$ with ${\rm{sg}}(\xi) = 0$ and ${\rm{bn}}(\xi) \leq 2$. 
	Arikan \cite{A} also classified those with ${\rm{sg}}(\xi) = 0$ and ${\rm{bn}}(\xi)=3$. 
	The standard contact structure $\xi_{st}$ on $S^3$ 
	is the only Stein fillable contact structure on a homology $3$-sphere in their classification. 
	It is well-known after Eliashberg \cite{El} that $D^4$ is a unique Stein filling of $(S^3, \xi_{st})$.
	Therefore it is a natural question 
	whether there exists a homology $3$-sphere $M$ admitting a Stein fillable contact structure 
	$\xi$ with ${\rm{sg}}(\xi) = 0$ and ${\rm{bn}}(\xi)=4$ 
	and what is a Stein filling of $(M,\xi)$. 
	
	In this note, we give an answer to this question by 
	constructing a family of Stein fillable contact 3-manifolds with ${\rm{sg}}(\xi) = 0$ and ${\rm{bn}}(\xi)=4$ 
	whose Stein filling is not diffeomorphic to $D^4$. 
	
	Furthermore, these Stein fillings are of \textit{Mazur type}. 
	Mazur \cite{M} introduced a contractible $4$-manifold whose boundary is a homology $3$-sphere not diffeomorphic to $S^3$, 
	called the Mazur manifold. A Mazur type manifold was defined 
	as a generalization of the Mazur manifold 
	(for the precise definition, see Definition \ref{def: Mazur}). 
	
	Now we are ready to state the main theorem in this note.
	\begin{theorem}\label{main theorem}
	There exists a family of contact homology $3$-spheres $\{ (M_{n}, \xi_{n})\}_{n\geq 1}$ such that 
	\begin{enumerate}
	\item $M_{1}, M_{2}, \dots $ are mutually not diffeomorphic, 
	\item each contact structure $\xi_{n}$ is Stein fillable and supported by an open book with page a $4$-holed sphere, and 
	\item a Stein filling $X_{n}$ of $(M_{n},\xi_{n})$ is a Mazur type manifold. 
	\end{enumerate}
	\end{theorem}
	
	We have one more motivation for our work. 
	Many examples of \textit{corks}, which are compact contractible Stein $4$-manifolds admitting a nice involution, 
	are known to be of Mazur type. 
	Some of them work to detect an exotic $4$-manifold pair. See \cite{AM} and \cite{AY}. 
	Thus the Stein fillings of the contact manifolds in the above theorem are candidates for corks, 
	and we may construct an exotic pair. 
	
	This note is constructed as follows. In Section $2$, we review some definitions and properties of positive Lefschetz fibrations and
	the Casson invariant. 
	In Section $3$, we prove Theorem \ref{main theorem}. 
	The proof of this theorem is based on works of 
	Loi and Piergallini \cite{LP} and Akbulut and Ozbagci \cite{AO}, who 
	proved that, for any {\it{positive allowable Lefschetz fibration}} ({\it{PALF}}\,) $f:X\rightarrow D^2$, 
	there exists a Stein fillable contact structure $\xi$ on $M=\del X$ such that 
	$\xi$ is supported by the open book obtained from $f$, and 
	$X$ is a Stein filling of $(M,\xi)$. 
	Thus first we give appropriate ordered collection of mapping classes as a monodromy of each PALF 
	belonging to a family of PALFs. 
	Drawing a Kirby diagram of the PALF and performing Kirby calculus, we examine it. 
	Finally we calculate the Casson invariant of its boundary and finish the proof of Theorem \ref{main theorem}. 
	
	\begin{acknowledge}
    \rm{The author would like to express his deep gratitude to his advisor, Hisaaki Endo, for his encouragement and many useful suggestions. 
    He would also like to thank the participants in the handle friendship seminar for helpful discussions on $4$- and $3$-dimensional topology 
	especially Tetsuya Abe for suggesting some references to the Alexander polynomial of a ribbon knot.} 
    \end{acknowledge}

   \section{Preliminaries.}
	We first review positive Lefschetz fibrations and the Casson invariant. 
	For more about contact topology we refer the reader to \cite{E}, \cite{Ge}, and \cite{OS}. 
	
	Throughout this note we will work in the smooth category and consider the homology groups with integer coefficients.
	\subsection{Positive Lefschetz Fibrations.}
   Let $X$ be a compact, oriented, smooth 4-manifold and $B$ a compact, oriented, smooth 2-manifold. 
   
   \begin{definition}
	A map $f:X\rightarrow B$ is called a \textit{positive Lefschetz fibration} if there exist points $b_{1}, b_{2}, \dots , b_{m}$ in 
	$\textrm{Int}(B)$ such that
	\begin{enumerate}
		\item $f|f^{-1}(B-\{b_{1}, b_{2}, \dots , b_{m}\}):f^{-1}(B-\{b_{1}, b_{2}, \dots , b_{m}\}) \rightarrow B-\{b_{1}, b_{2}, \dots ,\\ b_{m}\}$
				is a fiber bundle over $B-\{b_{1}, b_{2}, \dots , b_{m}\}$ with fiber diffeomorphic to an oriented surface $F$,
		\item $b_{1}, b_{2}, \dots , b_{m}$ are the critical values of $f$ with a unique critical point $p_{i} \in f^{-1}( b_{i})$ of $f$ for each $i$, 
		\item for each $b_{i}$ and $p_{i}$, there are local complex coordinate charts with respect to the given orientations of $X$ and $B$ such that locally $f$ can 
		be written as $f(z_{1},z_{2}) = z_{1}^2 + z_{2}^2$, and
		\item no fiber contains a $(-1)$-sphere, that is, an embedded sphere with self-intersection number $-1$.
	\end{enumerate}
   \end{definition}
   
	We call a fiber $f^{-1}(b)$ a \textit{singular fiber} if $b\in \{b_{1}, b_{2}, \dots , b_{m}\}$ or else a \textit{regular fiber}. 
	Also we call $X$ the \textit{total space} and $B$ the \textit{base space}.
	
	Here we will review roughly a handle decomposition of the total space of a given positive Lefschetz fibration over $D^2$. 
	For more details we refer the reader to \cite[$\S 8. 2$]{GS} and \cite[$\S 10. 1$]{OS}. 
	Suppose $f:X\rightarrow D^2$ is a positive Lefschetz fibration with fiber diffeomorphic to 
	a compact, oriented, connected genus $g$ surface $F$ with $r$ boundary components. 
	Then $X$ admits a handle decomposition 
	\begin{eqnarray}
		X & = &  (D^2 \times F) \cup (\bigcup_{i=1}^{m}h^{(2)}_{i}) \nonumber \\
		   & = &  (h^{(0)} \cup \bigcup_{j=1}^{2g+r-1} h^{(1)}_{j}) \cup (\bigcup_{i=1}^{m}h^{(2)}_{i}) \nonumber  
	\end{eqnarray}
	where each $h^{(k)}_{l}$ is a $k$-handle, and each $2$-handle $h_{i}^{(2)}$ corresponding to the critical point $p_{i}$ is 
	attached along a simple closed curve $\alpha_{i} \subset \{ pt. \} \times F \subset D^2 \times F$ with 
	framing $-1$ relative to the product framing on $\alpha_{i}$. 
	The attaching circle $\alpha_{i}$ of $h^{(2)}_{i}$ is called a \textit{vanishing cycle} for the singular fiber $f^{-1}(b_{i})$. 
	Then the ordered collection $( t_{\alpha_{1}}, \,  t_{\alpha_{2}} ,\cdots, t_{\alpha_{m-1}}, \, t_{\alpha_{m}})$ of positive Dehn twists 
	is called a \textit{monodromy} of $f$. 
	Conversely, note that a positive Lefschetz fibration over $D^2$ can be constructed from a given set of simple closed curves. 

	 We say that a positive Lefschetz fibration over $D^2$ is \textit{allowable} 
	if all of the vanishing cycles are homologically nontrivial in the fiber $F$. 
	 After this, a positive allowable Lefschetz fibration is abbreviated to a \textit{PALF}.
	 
	   \subsection{The Casson Invariant}
	   
	   Let $\mathcal{M}$ be the set of the homeomorphism classes of oriented homology $3$-spheres. We recall an invarinat of $\mathcal{M}$ 
	in this section.
	
	\begin{definition}
	A map $\lambda : \mathcal{M} \rightarrow \Z$ is called the {\textit{Casson invariant}} if 
		\begin{enumerate}
		\item $\lambda (S^3) = 0$ , 
		\item for any oriented homology $3$-sphere $M$, any knot $K$ in $M$, and any $m\in \Z$, 
				the difference
				\[
				\lambda (M + \frac{1}{m+1} \cdot K ) - \lambda (M + \frac{1}{m} \cdot K)
				\]
		is independent of $m$, where $M + (1/m) \cdot K$ is the manifold obtained from $M$ by $(1/m)$-surgery on $K$, and 
		\item for any oriented homology $3$-sphere $M$, any two components boundary link $K \cup K'$ in $M$, that is, 
		two knots $K$ and $K'$ which bound disjoint Seifert surfaces in $M$, 
				and any $m, n \in \Z$, 
				\begin{eqnarray}
				 && \lambda (M + \frac{1}{m+1} \cdot K +\frac{1}{n+1} \cdot K') 
						- \lambda (M + \frac{1}{m} \cdot K+\frac{1}{n+1} \cdot K') \nonumber \\
				&& -\lambda (M + \frac{1}{m+1} \cdot K +\frac{1}{n} \cdot K') + \lambda (M + \frac{1}{m} \cdot K+\frac{1}{n} \cdot K') \nonumber \\
				&& = 0 \, .\nonumber
				\end{eqnarray}
		\end{enumerate}
	\end{definition}
	
	\begin{remark}\label{rem}
		The existence and uniqueness up to sign of the Casson invariant were proven by Casson. 
		He also showed that this invariant $\lambda$ has some properties, in particular 
		$\lambda (M + ({1}/{(m+1)}) \cdot K ) - \lambda (M + ({1}/{m}) \cdot K) = ({1}/{2}) \Delta ''_{K\subset M}(1)$ 	
		for any oriented homology $3$-sphere $M$, any knot $K$ in $M$, and any $m\in \Z$, 
		where $\Delta _{K\subset M} (t)$ is the normalized Alexander polynomial of $K$ in $M$ 
		(see \cite{Sa1} and \cite[section 3.4]{Sa2} for other properties).
	\end{remark}
	   
	   In a special case, we can compute the Casson invariant of a given homology $3$-sphere by using the formula below. 
	
	\begin{proposition}[The surgery formula]\label{surgeryformula}
		Let $M$ be an oriented homology $3$-sphere and $K$ a knot in $M$. 
		Then, for any $m\in \Z$
		\[
		\lambda (M + \frac{1}{m} \cdot K) =  \lambda (M)+ \frac{m}{2} \Delta ''_{K\subset M}(1) \, .
		\]
	\end{proposition}
	
	\begin{proof}
	This proposition follows immediately from Remark \ref{rem}.
	\end{proof}
	
	 \section{Main theorem.}
	
	Let $S$ be a $4$-holed sphere and $\alpha$, $\beta$, and $\gamma$ simple closed curves in $S$ as in Figure \ref{fig: scc}. 
	We deal with a simple closed curve mapped by a diffeomorpsim of $S$ and it is complicated in general. 
	Thus we draw a simple closed curve in $S$ like as a graph for example in Figure \ref{fig: graph}. 
	This presentation can be justified by considering the original simple closed curve as one of the boundary components 
	of a tubular neighborhood of the graph. 
	
	 \begin{figure}[t]
			\begin{center}
				\includegraphics[width=150pt]{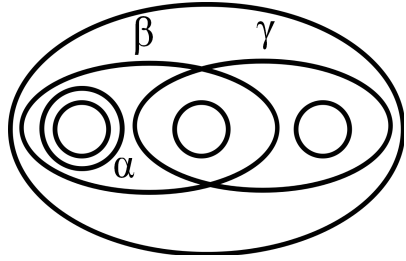}
				\caption{Simple closed curves $\alpha$, $\beta$, and $\gamma$ in $S$.}
				\label{fig: scc}
			\end{center}
	\end{figure}
	
    \begin{figure}[t]
			\begin{center}
				\includegraphics[width=300pt]{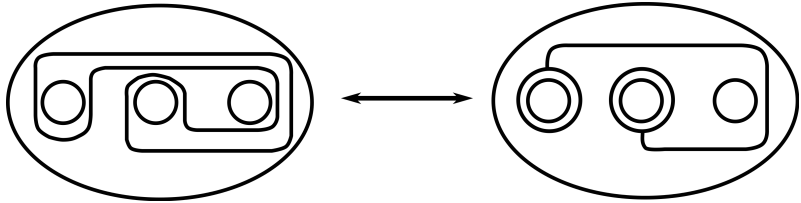}
				\caption{The graph on the right corresponds to the simple closed curve on the left.}
				\label{fig: graph}
			\end{center}
	\end{figure}
	
	Let $f_{n}: X_{n} \rightarrow D^2$ ($n\geq 1$) be a PALF with fiber $S$ whose 
	monodromy is 
	\[
	( t_{\alpha} \, , \, t_{\beta} \, , \, t_{ \gamma _{\, n}})
	\]
	where $\gamma _{n}$ is $(t_{\gamma} \, t_{\beta})^n (\gamma)$ (See Figure \ref{fig: X_{n}}).
	\begin{figure}[t]
		\begin{center}
			\includegraphics[width=300pt]{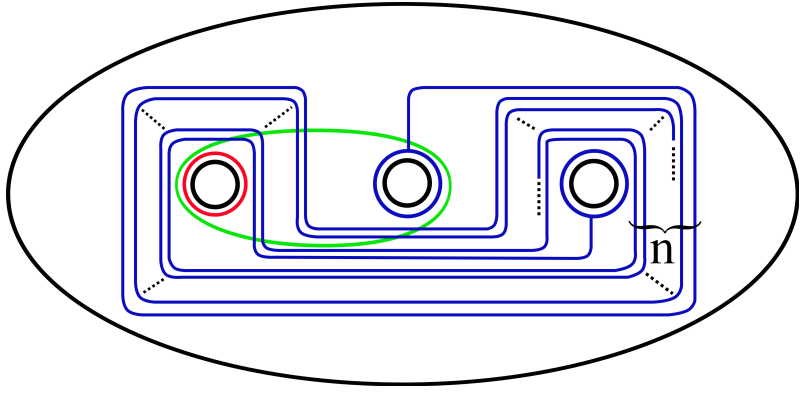}
			\caption{The vanishing cycles $\alpha$, $\beta$, and $\gamma_{n}$ of $f_{n}: X_{n} \rightarrow D^2$}
			\label{fig: X_{n}}
		\end{center}
	\end{figure}
	A Kirby diagram of $X_{n}$ can be drawn from the data of the vanishing cycles as the first diagram in Figure \ref{fig: Kirbydiagram}. 
	Performing Kirby calculus for this diagram as in Figure \ref{fig: Kirbydiagram}, 
	we obtain a simple diagram of $X_{n}$ as the last of Figure \ref{fig: Kirbydiagram}. 
	
    Before the proof of Theorem \ref{main theorem}, we define one more crucial notion. 
    
	\begin{definition}\label{def: Mazur}
	A $4$-manifold $X$ is called of \textit{Mazur type} if 
	it is contractible, the boundary $\del X$ is not diffeomorphic to $S^3$, and 
	it admits a handle decomposition consisting of one $0$-handle, one $1$-handle, and one $2$-handle.
	\end{definition}

	\begin{proof}[Proof of Thorem \ref{main theorem}]
		We first note that each $\del X_{n}$ admits the Stein fillable contact structure $\xi_{n}$ 
		supported by the open book obtained from $f_{n} : X_{n} \rightarrow D^2$ according to 
		Loi and Piergallini \cite{LP} and Akbulut and Ozbagci \cite{AO}. 
		We will show that $\{ (\del X_{n} , \xi_{n} )\}_{n\geq 1}$ is a family we want here.
		
		By the last diagram of $X_{n}$ as in Figure \ref{fig: Kirbydiagram}, 
		$H_{*}(X_{n}) = H_{*}(\{ pt. \})$ and $\pi_{1}(X_{n}) = 1$, 
		so each $X_{n}$ is contractible. $X_{n}$ has a handle decomposition satisfying the condition for a Mazur type manifold. 
		Thus once we show $\del X_{n}$ is not diffeomorphic to $S^3$, it follows that $X_{n}$ is of Mazur type. 
		Considering the first diagram of $X_{n}$ in Figure \ref{fig: surgerydiagram} as a surgery diagram of $\del X_{n}$ and 
		Performing some Kirby moves for the diagram of $\del X_{n}$ as in Figure \ref{fig: surgerydiagram}, we conclude that $\del X_{n}$ is 
		obtained from the Dehn surgery on the knot $K_{n}$ in Figure \ref{fig: surgerydiagram} with surgery coefficient $1$. 
			\begin{figure}[b]
		\begin{center}
				\includegraphics[width=400pt]{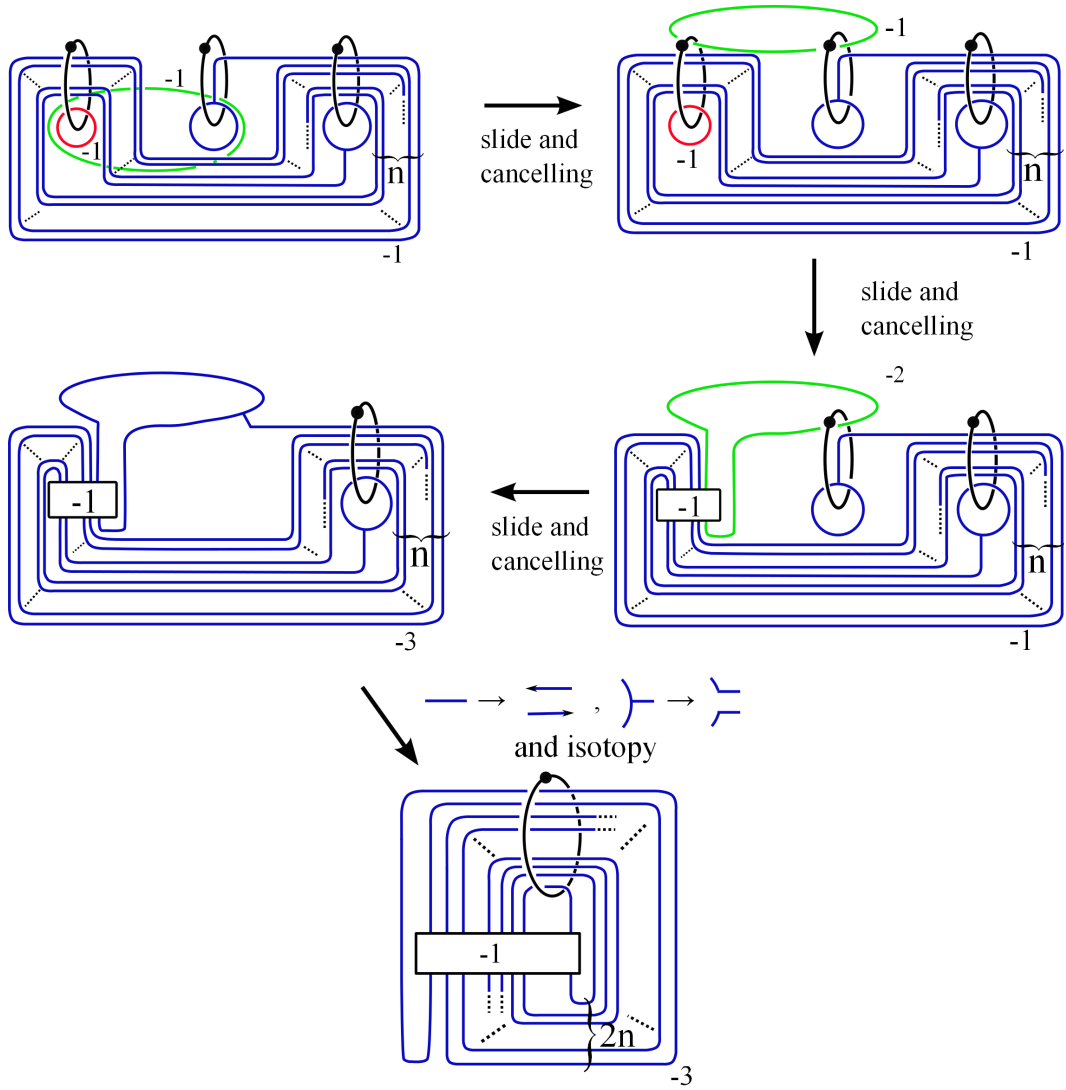}	
				\caption{Kirby diagrams of $X_{n}$ and Kirby calculus for the diagrams.}
				\label{fig: Kirbydiagram}
		\end{center}
		\end{figure}
		In order to determine the Casson invariant of $\del X_{n}$, we first compute the Alexander polynomial 
		$\Delta _{K_{n} \subset S^3}(t)$ i.e., $\Delta _{K_{n}}(t)$.
		Fox and Milnor \cite[Theorem $2$]{FM} showed that, if a knot $K$ is a slice knot, 
		there exists a polynomial $f(t)$ with integer coefficients such that $\Delta_{K}(t)=f(t) f(t^{-1})$.
		Furthermore it is well-known that, if $K$ is a ribbon knot, 
		the above polynomial $f(t)$ is equal to the Alexander polynomial $\Delta_{D}(t)$ of 
		a ribbon disk $D$ of $K$ up to units, and it is computed from 
		the fundamental group of the ribbon disk exterior 
		(cf. Terasaka \cite{Tera} computed $\Delta_{K}(t)$ of a ribbon knot $K$ from the Wirtinger presentation of $K$).
		Since each $K_{n}$ is a ribbon knot, considering the ribbon disk exterior $D^4 - \nu (D_{n})$ as in Figure \ref{fig: exterior}, 
		where $D_{n}$ is a ribbon disk of $K_n$ and $\nu(D_{n})$ is its tubular neighborhood in $D^4$, 
		we can determine the fundamental group of 
		$D^4 - \nu (D_{n})$ as follows ; 
		\[
		\pi_{1}( D^4 - \nu (D_{n})) = \langle x,y \, |\, (xy)^{n} x (xy)^{-n} y^{-1} \rangle . 
		\]
		Using this presentation of $\pi_{1}( D^4 - \nu (D_{n}))$, it follows from the above facts
		that 
		\[ 
			\Delta_{D_{n}}(t) = f(t) = 1- t + t^2 -t^3 + \cdots + t^{2n} ,
		\]
		and
		\begin{eqnarray}
			\Delta _{K_{n}} (t) & =  & f(t) f(t^{-1}) \nonumber \\
			& = & t^{2n} - 2 t^{2n-1} + 3 t^{2n-3} - \cdots - 2nt +(2n+1) \nonumber \\
			&   & -2nt^{-1} + (2n+1)t^{-2} \cdots - 2 t^{-(2n-1)} + t^{-2n}. \nonumber
		\end{eqnarray}
		 Thus 
		\[
		\Delta '' _{K_{n}} (1) = 2n(n+1) ,
		\]
		and by Proposition \ref{surgeryformula}
		\[
		\lambda (\del X_{n}) = \lambda (S^3 + {1} \cdot K_{n}) = \frac{1}{2}\times 2n(n+1)  = n (n+1) \neq 0. \
		\] 
		We conclude that $\del X_{1}, \del X_{2}, \dots$ are not diffeomorphic to $S^3$
		and mutually not diffeomorphic. Now we complete the proof of Theorem \ref{main theorem}.
		
\end{proof}

		\begin{figure}[h]	
					\includegraphics[width=380pt]{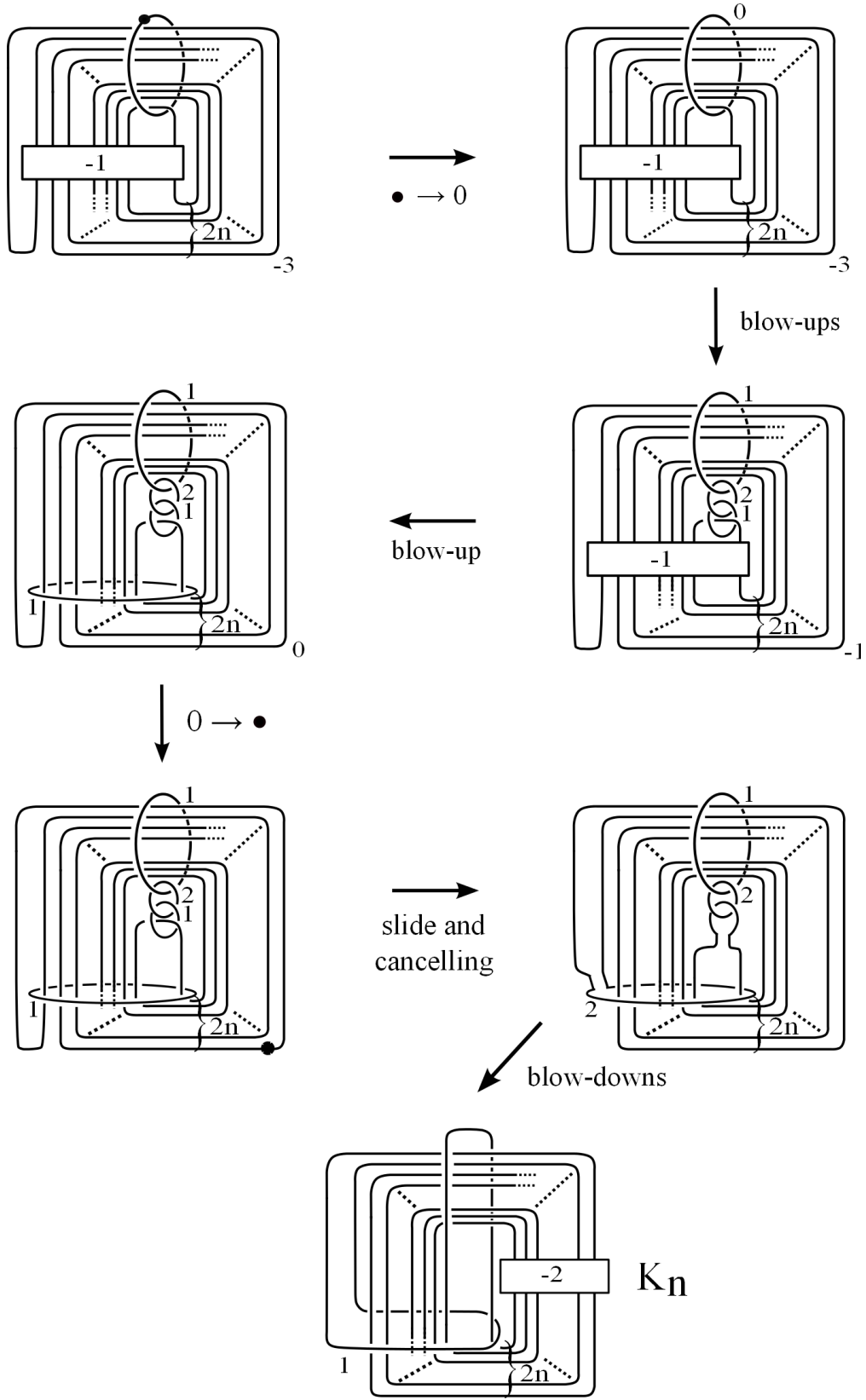}	
					\caption{Surgery diagrams of $\del X_{n}$ and deformation of the diagrams.}
					\label{fig: surgerydiagram}
		\end{figure} 
		\begin{figure}[h]
			\begin{center}
				\includegraphics[width=150pt]{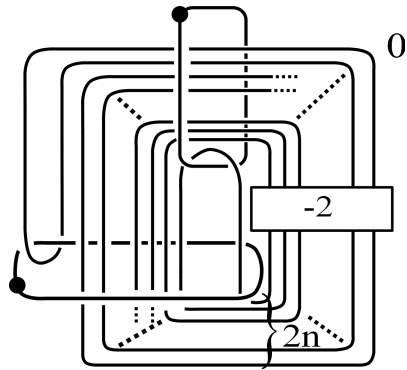}	
				\caption{The ribbon disk exterior for $K_{n}$.}
				\label{fig: exterior}
			\end{center}
		\end{figure}

\end{document}